\documentclass[11pt]{amsart}
\usepackage[utf8]{inputenc}
\usepackage{amsmath}
\usepackage{amsthm}
\usepackage{hyperref}
\usepackage{cleveref}
\usepackage{amssymb}
\usepackage[left=1.5in,right=1.5in]{geometry}
\usepackage{todonotes}
\usepackage{tikz}
\usepackage{tikz-cd}
\usepackage{mathtools}
\usepackage{enumerate}
\usepackage{amsfonts}
\usepackage{mathrsfs}

\newtheorem{theorem}{Theorem}
\numberwithin{theorem}{section}
\newtheorem{lemma}[theorem]{Lemma}
\newtheorem{proposition}[theorem]{Proposition}
\newtheorem{conjecture}[theorem]{Conjecture}
\newtheorem{corollary}[theorem]{Corollary}

\theoremstyle{definition}
\newtheorem{definition}[theorem]{Definition}
\newtheorem{example}{Example}

\theoremstyle{remark}
\newtheorem{remark}[theorem]{Remark}
\newtheorem{obs}[theorem]{Observation}

\Crefname{conjecture}{Conjecture}{Conjectures}

\newcommand{\cH}{\mathcal{H}}
\newcommand{\cA}{\mathcal{A}}
\newcommand{\cY}{\mathcal{Y}}

\newcommand{\tR}{\widetilde{R}}

\newcommand{\ZZ}{\mathbb{Z}}

\DeclareMathOperator{\car}{Car}

\DeclareMathOperator{\cross}{cr}

\newcommand{\nxt}{\mathrm{next}}
\newcommand{\prev}{\mathrm{prev}}

\usetikzlibrary{decorations.markings, arrows.meta}
\tikzcdset{
 arrow style=tikz,
 diagrams={>={Straight Barb[scale=0.8]}},
 marrow/.style={decoration={markings,mark=at position 0.5 with {\arrow{#1}}}, postaction=decorate},
 rmarrow/.style={decoration={markings,mark=at position 0.8 with {\arrow{#1}}}, postaction=decorate}
}

\title[The BBDVW Conjecture for KL polynomials of lower intervals]{The BBDVW Conjecture for Kazhdan--Lusztig polynomials of lower intervals}
\author{Grant T. Barkley}
\author{Christian Gaetz}
\address[Barkley]{Department of Mathematics, Harvard University, Cambridge, MA.}
\email{{\href{mailto:gbarkley@math.harvard.edu}{gbarkley@math.harvard.edu}}}
\address[Gaetz]{Department of Mathematics, University of California, Berkeley, CA.}
\email{{\href{mailto:gaetz@berkeley.edu}{{\tt gaetz@berkeley.edu}}}}
\date{\today}

\begin{document}
\begin{abstract}
Blundell, Buesing, Davies, Veličković, and Williamson (BBDVW) introduced the notion of a hypercube decomposition of an interval in Bruhat order. They conjectured a recursive formula in terms of this structure which, if shown for all intervals, would imply the Combinatorial Invariance Conjecture of Lusztig and Dyer, for Kazhdan--Lusztig polynomials of the symmetric group. In this article, we prove implications between the BBDVW Conjecture and several other recurrences for hypercube decompositions, under varying hypotheses, which have appeared in the recent literature. As an application, we prove the BBDVW Conjecture for \emph{lower} intervals $[e,v]$, the first non-trivial class of intervals for which it has been established. 
\end{abstract}
\maketitle
\section{Introduction}

Kazhdan--Lusztig polynomials \cite{Kazhdan-Lusztig-1} have long been studied for their deep connections to Hecke algebras, the geometry of Schubert varieties, and representation theory. Given two elements $u$ and $v$ of a Coxeter group $W$, the Kazhdan--Lusztig polynomial $P_{u,v}(q) \in \mathbb{Z}[q]$ can be defined via a recurrence arising in the computation of the canonical basis for the Hecke algebra associated to $W$; see \Cref{sec:kl-polynomials}. When $W$ is the Weyl group of an algebraic group, the coefficients of $P_{u,v}$ give the dimensions of cohomology groups of a stalk of the intersection cohomology complex on a Schubert variety. Lusztig and Dyer independently made the following surprising conjecture.

\begin{conjecture}[Lusztig c.~1983; Dyer \cite{Dyer1987}]\label{conj:combinatorial}
    If the Bruhat intervals $[u,v]\subset W$ and $[u',v']\subset W$ are isomorphic posets, then $P_{u,v}=P_{u',v'}$.
\end{conjecture}

This \emph{Combinatorial Invariance Conjecture} has received significant attention. The conjecture has been proven in the case $u=u'=e$ are the identity element of $W$ \cite{Ducloux, Brenti-lower-intervals, Lower-intervals-general-type, Delanoy}, in the case where $\ell(u,v)$ is small and $W$ is finite \cite{Incitti, Incitti2007,Esposito1,Esposito2}, and for the group $W=\widetilde{A}_2$ \cite{affine-A2}. It has also been shown in finite simply laced type that, if $[u,v] \cong [u',v']$, then $P_{u,v}$ and $P_{u',v'}$ have the same linear coefficient \cite{Patimo-q-coefficient}. It is known \cite{parabolic-paper} that \Cref{conj:combinatorial} implies several related conjectures in the parabolic setting, including the main conjecture of \cite{Marietti-sm-for-parabolic,Marietti-parabolic-cic}. See \cite{Brenti-open-problems} for a discussion of other known results.

Recently, the case $W=S_n$ (the most studied case) was treated by Blundell, Buesing, Davies, Veličković, and Williamson \cite{Blundell}. Using techniques from machine learning \cite{davies}, they formulated a conjectural recurrence for the $P_{u,v}$, which uses only the poset structure of $[u,v]$. Their formula depends on a special kind of order ideal $I \subset [u,v]$ called a \emph{hypercube decomposition}. Given a hypercube decomposition $I$, they define two polynomials, $N_{u,v,I}$ and $Q_{u,v,I}$, which depend only on the Kazhdan--Lusztig polynomials of smaller subintervals; see \Cref{sec:bbdvw-polynomials}\footnote{In \Cref{sec:formula}, we also give a simpler, equivalent, expression for the polynomial $N_{u,v,I}$.}.

\begin{conjecture}[\cite{Blundell}] 
\label{conj:BBDVW}
Let $I$ be any hypercube decomposition of $[u,v] \subset S_n$. Then
\begin{equation}
\label{eq:bbdvw-intro}
         q^{\ell(u,v)}P_{u,v}(q^{-1})-P_{u,v}(q) = N_{u,v,I} + Q_{u,v,I}. 
\end{equation}
\end{conjecture}

If true, this conjecture gives an algorithm for computing $P_{u,v}$ given only the poset underlying $[u,v]$: find a proper hypercube decomposition of $[u,v]$, recursively compute $N_{u,v,I}$ and $Q_{u,v,I}$, and use the known degree bound to compute $P_{u,v}$. In \cite{Blundell} it was shown that this algorithm always terminates, since every interval of a symmetric group has a proper hypercube decomposition, the \emph{standard hypercube decomposition}. Thus \Cref{conj:BBDVW} implies \Cref{conj:combinatorial}. It was also shown in \cite{Blundell}, using techniques from geometric representation theory, that the standard hypercube decomposition satisfies \Cref{conj:BBDVW}. This is not enough to deduce Combinatorial Invariance, because isomorphic Bruhat intervals may have non-isomorphic standard hypercube decompositions. 
 
We say that $[u,v]$ is a \emph{lower interval} if $u=e$ is the identity element of $S_n$.
Our main theorem establishes \Cref{conj:BBDVW} for lower intervals.

\begin{theorem}\label{thm:BBDVWorig}
    Let $I$ be any hypercube decomposition of a lower interval $[e,v] \subset S_n$. Then 
    \[ q^{\ell(e,v)}P_{e,v}(q^{-1})-P_{e,v}(q) = N_{e,v,I} + Q_{e,v,I}. \]
\end{theorem}

Although \cite{Blundell} has inspired a surge of work on hypercube decompositions and the Combinatorial Invariance Conjecture (see \cite{elementary-paper, Brenti-Marietti, Gurevich-Wang, double-shortcut-paper}), \Cref{thm:BBDVWorig} is the first non-trivial class of intervals for which \Cref{conj:BBDVW} has been proven. These other papers use hypercube decompositions to study the $R$- and $\tR$-polynomials, for which \Cref{conj:combinatorial} can equivalently be stated, but it has not been clear heretofore how these results are related to the BBDVW \Cref{conj:BBDVW}.

These other works on hypercube decompositions have also required the imposition of additional hypotheses on hypercube decompositions, beyond those originally appearing in \cite{Blundell}. The formulas for \emph{relative $R$-polynomials} in \cite{Gurevich-Wang} are considered only relative to special hypercube decompositions which generalize the standard hypercube decomposition, but which are not combinatorial invariants of an interval. In \cite{elementary-paper} we considered the combinatorial notion of a \emph{strong} hypercube decomposition (see \Cref{def:strong-nc-join-E}(\ref{item:def-strong})) as well as \emph{property (E)} (see \Cref{def:strong-nc-join-E}(\ref{item:def-E})). And in \cite{Brenti-Marietti,double-shortcut-paper} the results require \emph{meet}\footnote{The conventions of \cite{Brenti-Marietti} use \emph{upper} hypercube decompositions which are poset-dual to the (lower) hypercube decompositions used here and in \cite{Blundell}. Therefore what is called a \emph{join} hypercube decomposition in \cite{Brenti-Marietti} is most naturally called a meet hypercube decomposition here.} and \emph{amazing}  hypercube decompositions, respectively. In \Cref{def:strong-nc-join-E}(\ref{item:def-nc}) we introduce a new hypothesis, the \emph{numerical criterion}. 

A hypercube decomposition $I \subset [u,v]$ determines the \emph{hypercube maps} $\theta_x: \cA_x \to [u,v]$, where $\cA_x$ denotes the set of antichains of the poset $\cY_x$ of elements in $[u,v]\setminus I$ differing from $x \in I$ by a reflection (see \Cref{sec:hcd-background}). For any order ideal $I \subset [u,v]$ we introduce the \emph{relative $\tR$-polynomial} $\tR_{u,v,I}$ (see \Cref{def:relative-tR}). This is (a variable transform of) a generalization of the \emph{relative $R$-polynomial} of Gurevich--Wang \cite{Gurevich-Wang}. We show that, given strongness and property (E), the relative $\tR$-polynomial has a simple formula in terms of the hypercube map. 

\begin{theorem}
\label{thm:intro-relative-R}
Let $[u,v]$ be an interval in the symmetric group. Let $I$ be any diamond-closed order ideal satisfying property (E) which has a strong hypercube cluster at $u$. Then 
\[ \tR_{u,v,I} = \sum_{\substack{Y\in\cA_u \\ \theta_u(Y)=v}} q^{|Y|}. \]
\end{theorem}

In addition to proving the BBDVW formula, we will show that, under the same hypotheses, a simpler version of the BBDVW formula holds for the $R$-polynomial. The notation used in this Theorem is introduced in \Cref{sec:prelim}.

\begin{theorem}\label{thm:introsimple}
    If $[u,v]$ is a simple interval, and $I$ is any \emph{strong} hypercube decomposition of $[u,v]$, then
    \[ \sum_{Y\in \cA_u} (-q)^{|Y|} \tR_{\theta_u(Y),v} = 0. \]
    If, moreover, $I$ satisfies the numerical criterion, then (\ref{eq:bbdvw-intro}) holds.
\end{theorem}
\noindent 
The following conjecture is a natural strengthening of \Cref{thm:introsimple}, and would imply the Combinatorial Invariance Conjecture.

\begin{conjecture}\label{conj:Feq0}
    Let $I$ be a strong hypercube decomposition of $[u,v]$ satisfying the numerical criterion at $u$, then:
    \[ \sum_{Y\in \cA_u} (-q)^{|Y|}\tR_{\theta_u(Y),v} = 0. \] 
\end{conjecture}

We prove that, at least for lower intervals, some of these hypotheses are automatically satisfied.
 
\begin{theorem}
\label{thm:intro-strong-and-nc}
Let $I$ be any hypercube decomposition of a lower interval $[e,v] \subset S_n$. Then $I$ has property (E) and the hypercube cluster at $e$ is strong and satisfies the numerical criterion. 
\end{theorem}

We also pose the following conjecture, which involves the formula for $\tR$-polynomials introduced in \cite{elementary-paper}. It, too, would imply the Combinatorial Invariance Conjecture.
\begin{conjecture}\label{conj:ReqH}
    Let $I$ be a hypercube decomposition of $[u,v]$ satisfying the numerical criterion, then
    \[ \tR_{u,v} = \sum_{x\in I}\sum_{\substack{Y\in\cA_x\\ \theta_x(Y)=v}} q^{|Y|}\tR_{u,x}. \]
\end{conjecture}

\subsection*{Outline}
In \Cref{sec:prelim} we give preliminaries on Kazhdan--Lusztig polynomials, hypercube decompositions, and the polynomials appearing the BBDVW \Cref{conj:BBDVW}. 

In \Cref{sec:newN}, we reformulate the BBDVW Conjecture in terms of $R$-polynomials. In \Cref{sec:new-hypotheses} we introduce the hypotheses of strongness, the numerical criterion, and property (E) for hypercube clusters. In \Cref{sec:relative-R}, we introduce the relative $\tR$-polynomials $\tR_{u,v,I}$ for order ideals $I \subset [u,v]$ and prove \Cref{thm:intro-relative-R}, showing that when $I$ admits a hypercube cluster satisfying strongness and property (E), the relative $\tR$-polynomials can be computed in terms of the hypercube map $\theta$. In \Cref{sec:relative-R-to-BBDVW} we show that if $I$ additionally satisfies the numerical criterion, then the reformulated BBDVW Conjecture holds for $I$. And in \Cref{sec:property-E} we show that hypercube decompositions of simple intervals always satisfy property (E) and use this to prove \Cref{thm:introsimple}.

In \Cref{sec:lower} we study structural properties of the hypercube decompositions of lower intervals in order to prove \Cref{thm:intro-strong-and-nc}. Together with the results of \Cref{sec:formula}, this implies \Cref{thm:BBDVWorig}, resolving the BBDVW Conjecture for lower intervals.

\section{Preliminaries}\label{sec:prelim}

We refer the reader to \cite{Bjorner-Brenti} for background on Coxeter groups.

\subsection{Kazhdan--Lusztig polynomials}
\label{sec:kl-polynomials}
Let $W$ be a Coxeter group with simple generating set $S$ and length function $\ell$. The \emph{reflections} of $W$ are the conjugates of elements of $S$. The \emph{Bruhat graph} of $W$ is the directed graph  $\Gamma$ with vertex set $W$ and edges $x\to y$ whenever $yx^{-1}$ is a reflection and $\ell(x)<\ell(y)$. We sometimes view $\Gamma$ as an edge-labeled graph, where an edge $x\to y$ is labeled by the reflection $yx^{-1}$. The \emph{Bruhat order} is the partial order on $W$ so that $x\leq y$ if and only if there is a path from $x$ to $y$ in the Bruhat graph. We write $\Gamma(u,v)$ for the induced subgraph of $\Gamma$ on the elements of $[u,v]$.
\begin{theorem}[Kazhdan--Lusztig \cite{Kazhdan-Lusztig-1}]
    There is a unique family of polynomials $(R_{u,v})_{u,v\in W}$ so that for any $u,v\in W$, the following hold:
    \begin{itemize}
        \item If $u\not\leq v$, then $R_{u,v}=0$.
        \item If $u=v$, then $R_{u,v}=1$.
        \item If $u<v$, and $s\in S$ is such that $sv<v$, then
        \[ R_{u,v} = \begin{cases}
            R_{su,sv} &\text{if $su<u$} \\ 
            (q-1)R_{u,sv} + qR_{su,sv} &\text{if $su>u$}
        \end{cases}. \]
    \end{itemize}
\end{theorem}

From the recurrence, one can deduce that there is a unique polynomial $\tR_{u,v}$ so that $R_{u,v}(q)= q^{\frac{\ell(u,v)}{2}}\tR_{u,v}(q^{1/2}-q^{-1/2})$.
Matthew Dyer introduced a description of $\tR$-polynomials using only the edge-labeled graph $\Gamma(u,v)$. His description uses reflection orders; for simplicity, we define these only for $W=S_n$. In this case, reflections coincide with the transpositions.
\begin{definition}[\cite{Dyer1987}]
    A \emph{reflection order} for $S_n$ is a total ordering $\prec$ on the set of reflections so that for all $1\leq a < b < c \leq n$, either 
    \[ (a\,b) \prec  (a\,c) \prec (b\,c) \qquad\text{or}\qquad (b\,c) \prec (a\,c) \prec (a\,b). \]
    We say that a path 
    \[ x_0 \xrightarrow{t_1} x_1 \xrightarrow{t_2} \cdots \xrightarrow{t_k} x_k \]
    in $\Gamma$ is \emph{increasing} (with respect to $\prec$) if $t_1\prec t_2\prec \cdots\prec  t_k$. We denote the set of increasing paths from $u$ to $v$ by $\Gamma_\prec(u,v)$, and the length $k$ of such a path $\gamma$ by $\ell(\gamma)$. We say a path $\gamma$ is increasing at $x_i$ if $t_i\prec t_{i+1}$ and we say $\gamma$ is increasing from $x_i$ to $x_j$ if $t_{i+1}\prec \cdots \prec t_j$. We similarly define \emph{decreasing}. 
\end{definition}

\begin{theorem}[Dyer \cite{Dyer1987}]
\label{thm:dyer-increasing-paths}
    For any $u,v \in W$, we have that
    \[ \tR_{u,v}(q) = \sum_{\gamma\in \Gamma_\prec(u,v)} q^{\ell(\gamma)}. \]
\end{theorem}

In particular, the sum $\sum_{\gamma\in \Gamma_\prec(u,v)} q^{\ell(\gamma)}$ is independent of the choice of reflection order $\prec$.

\begin{theorem}[Kazhdan--Lusztig \cite{Kazhdan-Lusztig-1}] \label{thm:KLexistence}
    There is a unique family of polynomials $(P_{u,v})_{u,v\in W}$ so that for any $u,v\in W$, the following hold:
    \begin{itemize}
        \item If $u\not\leq v$, then $P_{u,v}=0$.
        \item If $u=v$, then $P_{u,v}=1$.
        \item If $u<v$, then $\deg P_{u,v} < \frac{1}{2}\ell(u,v)$, and
        \[ q^{\ell(u,v)}P_{u,v}(q^{-1})-P_{u,v}(q) = \sum_{x\in (u,v]} R_{u,x}P_{x,v}. \]
    \end{itemize}
\end{theorem}

\subsection{Hypercube decompositions}
\label{sec:hcd-background}
From now on, we let $W=S_n$ be a symmetric group, with standard choice of Coxeter generators $s_i=(i \: i+1)$ for $i=1,\ldots,n-1$. The definitions in this section were introduced in \cite{Blundell} and \cite{elementary-paper}; see those references for more information.

\begin{definition}
    Let $\cH_n$ denote the \emph{hypercube graph}, which is the Hasse diagram of a Boolean lattice with $n$ elements, viewed as a directed graph. An \emph{$n$-hypercube} in $\Gamma(u,v)$ is a subgraph isomorphic to $\cH_n$. A $2$-hypercube in $\Gamma(u,v)$ is called a \emph{diamond}. A subset $I$ of $[u,v]$ is \emph{diamond-closed} if, whenever $I$ contains three vertices of a diamond $\cH$ in $\Gamma(u,v)$, then $I$ contains the fourth vertex of $\cH$.
    
    If $y_1,\ldots,y_n$ are the endpoints of the edges outgoing from the bottom vertex of an $n$-hypercube $\cH$ in $\Gamma(u,v)$ and terminating in $\cH$, then we say that $\cH$ is \emph{spanned} by $\{y_1,\ldots,y_n\}$.
\end{definition}
In general, there may be multiple hypercubes spanned by a given set $Y$.

\begin{definition}
\label{def:cluster}
    Let $I$ be a subset of $[u,v]$ and let $x\in I$. We define the sets
    \begin{align*}
        \cY_x &\coloneqq \{ y\in [u,v] \setminus I \mid x\to y\} \\ 
        \cA_x &\coloneqq \{ Y\subseteq \cY_x \mid \text{$Y$ is an antichain}  \}. 
    \end{align*}
    
    We say there is a \emph{hypercube cluster} at $x$ (relative to $I$) if, for every $Y\in \cA_x$, there is a unique hypercube in $\Gamma(u,v)$ spanned by $Y$ and with bottom vertex $x$. In this case, we define a function $\theta_x:\cA_x\to [u,v]$, the \emph{hypercube map}, which sends an antichain $Y$ to the top vertex of the unique hypercube spanned by $Y$.

    If there is a hypercube cluster at $x$, then we extend the domain of $\theta_x$ to include arbitrary subsets $Y \subset \cY_x$ by defining $\theta_x(Y)$ to be the top of the unique hypercube spanned by the maximal elements of $Y$.
\end{definition}

When there is a hypercube cluster at $x$, we note that $\theta_x$ satisfies $\theta_x(\varnothing) = x$ and $\theta_x(\{y\}) = y$ for any $y\in \cY_x$.

\begin{definition}[\cite{Blundell}]
\label{def:hcd}
    A \emph{hypercube decomposition} of $[u,v]$ is a subset $I$ of $[u,v]$ such that the following hold:
    \begin{itemize}
        \item[(HD1)] $I=[u,z]$ for some $z \in [u,v]$, and
        \item[(HD2)] $I$ is diamond-closed in $[u,v]$, and
        \item[(HD3)] For every $x\in I$, there is a hypercube cluster at $x$ relative to $I$.
    \end{itemize}
\end{definition}

\subsection{BBDVW polynomials}
\label{sec:bbdvw-polynomials}
Let $I$ be a hypercube decomposition of $[u,v]$. In this section we will introduce the polynomials $N_{u,v,I}$ and $Q_{u,v,I}$ used in \Cref{thm:BBDVWorig}. Note that our $N_{u,v,I}$ is $(q-1)$ times the polynomial $I_{u,v,I}$ from \cite{Blundell}, and our $Q_{u,v,I}$ is $(q-1)$ times the polynomial $Q_{u,v,I}$ from \cite{Blundell}.

\subsubsection{$N_{u,v,I}$}
We recall the definition of $N_{u,v,I}$ from \cite{Blundell}. In \Cref{sec:newN}, we will show $N_{u,v,I}$ simplifies greatly when expressed using $R$-polynomials. 
We will exclusively use the simplified version in the proofs, so 
the auxiliary definitions in this subsection may be skipped by the reader uninterested in the comparison between the two definitions of $N_{u,v,I}$.

Let $M_I$ denote the free $\ZZ[q]$-module with basis $\{\delta_x\mid x\in I\setminus\{u\}\}$. There is another basis for $M_I$, with elements $b_y$ for $y\in I\setminus\{u\}$, defined by
\[ b_y =  \sum_{x\in (u,y]} P_{x,y}\cdot \delta_x. \]
Define 
\[ r_{u,v,I} \coloneqq \sum_{x\in I\setminus \{u\}} P_{x,v}\cdot \delta_x \]
and expand this element of $M_I$ in the basis $\{b_y\}_{y\in I\setminus\{u\}}$ to get some family of coefficients $(\gamma_y)_{y\in I\setminus \{u\}}$:
\[ r_{u,v,I} = \sum_{y\in I\setminus\{u\}} \gamma_y\cdot b_y. \]
Using these coefficients, define
\[ N_{u,v,I} \coloneqq \sum_{y\in I\setminus \{u\}} \gamma_y (q^{\ell(u,y)}P_{u,y}(q^{-1})-P_{u,y}(q)). \]

\subsubsection{$Q_{u,v,I}$}
Recall that $I$ is a hypercube decomposition of $[u,v]$, so there is a hypercube map $\theta_u$ sending subsets of $\cY_u$ to elements of $[u,v]$. Define
\begin{align}
\label{eq:original-Q}
Q_{u,v,I} \coloneqq &-q^{\ell(u,v)} \sum_{\varnothing \neq Y\subseteq \cY_u} (q^{-1}-1)^{|Y|} P_{\theta_u(Y),v}(q^{-1}) \\ \nonumber
= &- \sum_{\varnothing \neq Y\subseteq \cY_u} q^{\ell(u,\theta_u(Y))}(q^{-1}-1)^{|Y|} \sum_{y\in [\theta_u(Y),v]}R_{\theta_u(Y),y}P_{y,v}.
\end{align}

\section{An alternating sum formula}
\label{sec:formula}

\subsection{New expressions for the BBDVW polynomials}
\label{sec:newN}
We now show how to rewrite $N_{u,v,I}$ in terms of $R$-polynomials. We note that, by definition, $(\gamma_y)_{y\in I\setminus\{u\}}$ is the unique family of polynomials so that, for all $x\in I\setminus \{u\}$,
\[ P_{x,v} = \sum_{y\in I} \gamma_y P_{x,y}. \]

Now we rewrite, using \Cref{thm:KLexistence}:
\begin{align*}
    N_{u,v,I} &= \sum_{y\in I\setminus \{u\}} \gamma_y(q^{\ell(u,y)}P_{u,y}(q^{-1})-P_{u,y}) \\ 
    &= \sum_{y\in I\setminus \{u\}} \gamma_y \sum_{x\in (u,y]} R_{u,x}P_{x,y} \\
    &= \sum_{x\in I\setminus \{u\}} R_{u,x} \sum_{y\in I\setminus \{u\}} \gamma_yP_{x,y} \\ 
    &= \sum_{x\in I\setminus \{u\}} R_{u,x} P_{x,v}.
\end{align*}

So we may alternatively take $N_{u,v,I}\coloneqq \sum_{x\in I\setminus\{u\}} R_{u,x}P_{x,v}$ as the definition of $N_{u,v,I}$. This definition makes sense when $I$ is any order ideal of $[u,v]$.

For $Y\in\cY_u$, let $\Lambda(Y)$ denote the order ideal of $\cY_u$ generated by $Y$ (i.e. the set of elements of $\cY_u$ weakly below some element of $Y$ in Bruhat order). By partitioning the summation in (\ref{eq:original-Q}) into sums over those subsets of $\cY_u$ having the same antichain of maximal elements (and hence the same image under $\theta_u$), we find an expression for $Q_{u,v,I}$ as a sum over antichains:
\begin{align*} Q_{u,v,I} &= \sum_{\varnothing \neq Y\in \cA_u} q^{\ell(u,v)-|\Lambda(Y)|} (-1)^{|Y|+1}(q-1)^{|Y|}P_{\theta_u(Y),v}(q^{-1})  \\ 
&= \sum_{\varnothing\neq Y\in \cA_u} q^{\ell(u,\theta_u(Y))-|\Lambda(Y)|}(-1)^{|Y|+1}(q-1)^{|Y|}\sum_{y\in [\theta_u(Y),v]} R_{\theta_u(Y),y}P_{y,v}.
\end{align*}

We now give an equivalent form of \Cref{conj:BBDVW} that will be the form we prove. As originally stated, \Cref{conj:BBDVW} asserts
 \[   q^{\ell(u,v)}P_{u,v}(q^{-1}) - P_{u,v} = N_{u,v,I} + Q_{u,v,I}. 
   \]
   Expanding both sides, this is
   \[  
    \sum_{x\in (u,v]} R_{u,x}P_{x,v}
    = \sum_{x\in I\setminus \{u\}} R_{u,x} P_{x,v} - \sum_{\varnothing \neq Y\subseteq \cY_u} q^{\ell(u,\theta_u(Y))}(q^{-1}-1)^{|Y|} \sum_{y\in [\theta_u(Y),v]}R_{\theta_u(Y),y}P_{y,v}.
\]
We can rewrite this equality as
\[ \sum_{Y\subseteq \cY_u} q^{\ell(u,\theta_u(Y))}(q^{-1}-1)^{|Y|}\sum_{y\in [u,v]\setminus I} R_{\theta_u(Y),y}P_{y,v} = 0. \]
In terms of antichains, this is
\[ \sum_{Y\in \cA_u} (-1)^{|Y|}q^{\ell(u,\theta_u(Y))-|\Lambda(Y)|}(q-1)^{|Y|}\sum_{y\in [u,v]\setminus I} R_{\theta_u(Y),y}P_{y,v} = 0. \tag{BBDVW}\label{eq:BBDVW} \]
We refer to (\ref{eq:BBDVW}) as the BBDVW formula; this is the equality we will prove in what follows.

\subsection{Strongness, the numerical criterion, and property (E)}
\label{sec:new-hypotheses}
We now define several additional properties that a hypercube decomposition may satisfy. 

\begin{definition}
\label{def:strong-nc-join-E}
Suppose that the interval $[u,v]$ has a hypercube cluster at $x$, relative to an order ideal $I$, with hypercube map $\theta_x: \cA_x \to [u,v]$. We say that this hypercube cluster:
\begin{enumerate}[(i)] 
    \item \label{item:def-strong} is \emph{strong} \cite{elementary-paper} if for any $Y_1,Y_2\in \cA_x$ with $|Y_1|=|Y_2|=|Y_1\cap Y_2|+1$, if there is a subgraph of $\Gamma(u,v)$ of the form
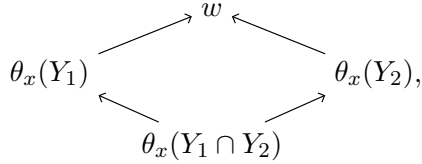
\[\begin{tikzcd}[cramped,sep=small] & w & \\ \theta_x(Y_1)\ar[ur]& &\theta_x(Y_2)\ar[ul] \\ &\theta_x(Y_1\cap Y_2)\ar[ul]\ar[ur]& \end{tikzcd},\]
then $Y_1\cup Y_2$ is an antichain and $\theta_x(Y_1\cup Y_2)=w$; 

    \item \label{item:def-nc} satisfies the \emph{numerical criterion} if for all $Y \in \cA_x$ we have
    \begin{equation}
    \label{eq:numerical-criterion}
        \frac{\ell(x,\theta_x(Y))+|Y|}{2} = |\Lambda(Y)|,
    \end{equation}
    where $\Lambda(Y)$ denotes the order ideal of the poset $(\cY_x,\leq)$ generated by $Y$; and

    \item \label{item:def-E} satisfies \emph{property (E)} \cite{elementary-paper} if there exists a reflection order $\prec$ so that if $x_1\xrightarrow{t_1} x_2$ and $x_3\xrightarrow{t_2} y$, with $x_1,x_2,x_3\in I\cap[x,v]$ and $y\in [x,v]\setminus I$, then $t_1\prec t_2$.
\end{enumerate}
We say that a hypercube decomposition $I$ of $[u,v]$ has one of these properties if the hypercube cluster at $x$ does, for all $x \in I$. 
\end{definition}

\subsection{The relative $\tR$-polynomial}
\label{sec:relative-R}
Let $I$ be a nonempty order ideal of $[u,v]$.  
\begin{definition}
\label{def:relative-tR}
    The \emph{relative $\tR$-polynomial} of $[u,v]$ with respect to $I$ is
    \[ \tR_{u,v,I} \coloneqq \sum_{x\in I} \tR_{u,x}(-q)\tR_{x,v}(q). \]
When $x,y\in [u,v]$, we also write $\tR_{x,y,I}$ for $\tR_{x,y,I\cap [x,y]}$.
\end{definition}

For a path $\gamma: x_0 \to \cdots \to x_k$ in $\Gamma(u,v)$, we let $\prev_{\gamma}(x_i)=x_{i-1}$ for $i=1,\ldots,k$ and $\nxt_\gamma(x_i)=x_{i+1}$ for $i=0,\ldots,k-1$. For an order ideal $I \subset [u,v]$, we write $\gamma|_{I}$ for the restriction of $\gamma$ to $I$. 

\begin{proposition}
\label{prop:relative-tR-gen-func}
If $I= [u,v]$, then we have 
\begin{equation}
\label{eq:rel-R-whole-interval}
   \tR_{u,v,I} = \begin{cases}
    1 &\text{if $u=v$}\\
    0 &\text{otherwise.}
\end{cases} 
\end{equation}

Otherwise,
let $I$ be a nonempty proper order ideal in $[u,v]$ and let $\prec$ be any reflection order. Then
\[
\tR_{u,v,I}=\sum_{\gamma \in \Gamma^I_{\prec}(u,v)} (-1)^{\ell(\gamma|_I)}q^{\ell(\gamma)}, 
\]
where $\Gamma^I_{\prec}(u,v)$ is the set of paths from $u$ to $v$ which are $\prec$-decreasing from $u$ to $\nxt(x)$ and $\prec$-increasing from $x$ to $v$, where $x$ is the last vertex on $\gamma|_I$.
\end{proposition}
\begin{proof}
By \Cref{thm:dyer-increasing-paths} and \Cref{def:relative-tR},
\[
\tR_{u,v,I}=\sum_{(\gamma_1,\gamma_2)} (-1)^{\ell(\gamma_1)}q^{\ell(\gamma_1)+\ell(\gamma_2)},
\]
where the sum ranges over pairs $(\gamma_1,\gamma_2)$, with $\gamma_1$ a decreasing path from $u$ to some point $x\in I$ and $\gamma_2$ an increasing path from $x$ to $v$. We can concatenate the pair $(\gamma_1,\gamma_2)$ to form a single path $\gamma$ from $u$ to $v$ having a marked point $x\in I$. There is a sign-reversing partial involution $\iota$ on the terms of $\tR_{u,v,I}$ which replaces the marked point $x$ with $\prev_{\gamma}(x)$ (resp. $\nxt_{\gamma}(x)$) if $\gamma$ is increasing (resp. decreasing) at $x$. For the purpose of this involution (and the others in this section), we declare every path to be decreasing at $u$ and increasing at $v$. The involution $\iota$ is well defined unless $x$ is the last vertex on $\gamma|_I$ and either $x=u=v$ or the path is decreasing at $x$. Hence by canceling terms paired by $\iota$, we see that $\tR_{u,v,I}$ counts paths $\gamma$ from $u$ to $v$ so that if $x$ is the last vertex on $\gamma|_I$, then the path is decreasing from $u$ to $\nxt(x)$ and increasing from $x$ to $v$, each with weight $(-1)^{\ell(\gamma|_I)}q^{\ell(\gamma)}$. 
\end{proof}

\begin{remark}
    Our relative $\tR$-polynomial is a generalization of that of Gurevich and Wang by \cite[Proposition 3.12]{Gurevich-Wang}. They give an interpretation of their relative $\tR$-polynomial in terms of Hecke algebras; it would be interesting to see if there is a similar description in our context. 
\end{remark}

\begin{figure}
    \centering
    \hfill
    \begin{tikzpicture}[baseline={([yshift=-.5ex]current bounding box.center)}]
        \fill[rounded corners, blue,opacity = .3] (0,-.5) -- (.5,0) --  (-1,1.5) -- (-1.5,1) -- cycle;
        \fill[red!20] (-1,1) circle (10pt); 
        \node (123) at (0,0) {$123$};
        \node (213) at (-1,1) {$x=213\phantom{=w}$};
        \node (132) at (1,1) {$132$};
        \node (312) at (1,2) {$312$};
        \node (231) at (-1,2) {$231$};
        \node (321) at (0,3) {$321$};

        \draw (123) -- (213) -- (312) -- (321) -- (231) -- (132) -- (123);
        \draw (213) -- (231);
        \draw (132) -- (312);
        \draw (123) -- (321);
        \draw[thick, red,->] (123) edge (213) (213) edge (231) (231) edge (321);
    \end{tikzpicture}
    \hfill $\xLeftrightarrow{\iota}$ \hfill
    \begin{tikzpicture}[baseline={([yshift=-.5ex]current bounding box.center)}]
        \fill[rounded corners, blue,opacity = .3] (0,-.5) -- (.5,0) --  (-1,1.5) -- (-1.5,1) -- cycle;
        \fill[red!20] (0,0) circle (10pt); 
        \node (123) at (0,0) {$\prev(x)=123\phantom{=\prev(w)}$};
        \node (213) at (-1,1) {$213$};
        \node (132) at (1,1) {$132$};
        \node (312) at (1,2) {$312$};
        \node (231) at (-1,2) {$231$};
        \node (321) at (0,3) {$321$};

        \draw (123) -- (213) -- (312) -- (321) -- (231) -- (132) -- (123);
        \draw (213) -- (231);
        \draw (132) -- (312);
        \draw (123) -- (321);
        \draw[thick, red,->] (123) edge (213) (213) edge (231) (231) edge (321);
    \end{tikzpicture}
    \hfill
    \phantom{!}
    
    \caption{An example of the involution $\iota$ from the proof of \Cref{prop:relative-tR-gen-func}. The order ideal $I$ is shown in blue. The reflection order is $(1\:2)\prec(1\:3)\prec(2\:3)$. The path $\gamma$ is shown with red arrows, and the marked point is highlighted in pink. In this example, $\gamma$ is increasing at $213$ and decreasing at $123$, so $\iota$ swaps the marked point between those two vertices.}
    \label{fig:involution}
\end{figure}

We show in \Cref{lem:R-from-relR,lem:Feq0relR} that the relative $\tR$-polynomial satisfies two identities related to \Cref{conj:Feq0,conj:ReqH}. The first of these is a generalization of \cite[Proposition 3.12]{Gurevich-Wang}. 

\begin{lemma}\label{lem:R-from-relR}
    Let $I$ be a nonempty order ideal of $[u,v]$. Then
    \[ \tR_{u,v} = \sum_{x\in I} \tR_{u,x}\tR_{x,v,I}. \]
\end{lemma}
\begin{proof}
    The right hand side counts paths $\gamma$ from $u$ to $v$ with marked points $x_1,x_2\in I$ so that $x_1\leq x_2$ are on the path $\gamma$, $x_2$ is the last point of $\gamma$ in $I$, and $\gamma$ is increasing from $u$ to $x_1$, decreasing from $x_1$ to $\nxt(x_2)$, and increasing from $x_2$ to $v$. The weight of $\gamma$ is $(-1)^{\ell(\gamma_1)}q^{\ell(\gamma)}$, where $\gamma_1$ is the subpath from $x_1$ to $x_2$. There is a sign-reversing involution which replaces the marked point $x_1$ with either $\prev(x_1)$ or $\nxt(x_1)$ depending on whether $\gamma$ is decreasing or increasing at $x_1$, respectively. This is well-defined unless $\gamma$ is increasing at $x_1$ and $x_1=x_2$. Hence the terms left uncancelled by the involution are simply the increasing paths from $u$ to $v$, with the marked points $x_1$ and $x_2$ both coinciding with the last point of $\gamma$ in $I$. These paths are counted with weight $q^{\ell(\gamma)}$, so the remaining terms sum to exactly $\tR_{u,v}$.
\end{proof}

\begin{lemma}\label{lem:Feq0relR}
    Let $I$ be a nonempty order ideal of $[u,v]$. Then
    \[ \sum_{y\in [u,v]} \tR_{u,y,I}(-q) \tR_{y,v}(q) = 0.  \]
\end{lemma}
\begin{proof}
    First we note that $\sum_{y\in I} \tR_{u,y,I}(-q)\tR_{y,v}(q) = \tR_{u,v}(q)$ by (\ref{eq:rel-R-whole-interval}). Hence it is enough to show that
    \[ \sum_{y\in [u,v]\setminus I} \tR_{u,y,I}(-q)\tR_{y,v}(q) =  -\tR_{u,v}(q). \]
    The left hand side of this counts paths $\gamma$ from $u$ to $v$ with marked points $x,y$ so that $x$ is the last point of $\gamma$ in $I$, $y$ is in $[u,v]\setminus I$, and $\gamma$ is decreasing from $u$ to $\nxt(x)$, increasing from $x$ to $y$, and decreasing from $y$ to $v$. The weight of such a path is $(-1)^{\ell(\gamma_2)}q^{\ell(\gamma)}$, where $\gamma_2$ is the subpath from $x$ to $y$. There is a sign-reversing involution which replaces the marked point $y$ with either $\prev(y)$ or $\nxt(y)$, depending on whether $\gamma$ is decreasing at $y$ or increasing at $y$, respectively. This is well-defined unless $\gamma$ is decreasing at $y$ and $y=\nxt(x)$. Hence the remaining terms count decreasing paths $\gamma$ with weight $-q^{\ell(\gamma)}$, so sum to $-\tR_{u,v}$.
\end{proof}

As a consequence of either of these lemmas, the $\tR$-polynomial of an interval can be computed in terms of the $\tR$-polynomials of strictly smaller intervals and the relative $\tR$-polynomials with respect to any proper non-empty order ideal. In particular, if every interval has a combinatorially defined order ideal for which the relative $\tR$-polynomial can be computed combinatorially, then combinatorial invariance holds. 

We now restate and prove \Cref{thm:intro-relative-R}.

\begin{theorem}\label{thm:EintervalrelR}
    Let $I$ be a diamond-closed order ideal of $[u,v]$. Assume there is a hypercube cluster at $x\in I$ which is strong and satisfies property (E). Then
    \[ \tR_{x,v,I} = \sum_{\substack{Y\in\cA_x \\ \theta_x(Y)=v}}q^{|Y|}. \]
\end{theorem}
\begin{proof}
    Recall that $\tR_{x,v,I}$ is a sum over paths $\gamma$ from $x$ to $v$ which are decreasing at every point in $I$ and increasing at every point in $[u,v]\setminus I$. Since the hypercube cluster at $x$ satisfies property (E), there is a reflection order $\prec$ so that the edge labels of $I$ in $[x,v]$ are smaller than the edge labels of edges leaving $I$ in $[x,v]$. Using this reflection order, if $x'$ is the last vertex of $\gamma$ in $I$ and $x'\neq x$, then $\gamma$ is increasing at $x'$. Thus no paths counted by $\tR_{x,v,I}$ use any vertex of $I$ other than $x$. In other words, with such a reflection order, $\tR_{x,v,I}$ counts increasing paths from $x$ to $v$ so that $x$ is the only element from $I$ on the path. Now \cite[Lemmas 3.12 and 3.15]{elementary-paper} together show that such paths are in bijection with elements $Y\in\cA_x$ so that $\theta_x(Y)=v$. Furthermore, the length of the path coincides with $|Y|$. Hence $\tR_{x,v,I}=\sum_{\substack{Y\in\cA_x\\ \theta_x(Y)=v}}q^{|Y|}$.
\end{proof}

\subsection{From relative $\tR$-polynomials to the BBDVW formula}
\label{sec:relative-R-to-BBDVW}
\begin{lemma}
\label{lem:bbdvw-inner-sum-vanishes}
    Assume that $I$ is a diamond-closed order ideal of $[u,v]$ with a hypercube cluster at $u$ that is strong and has property (E). Then for any $y\in [u,v]\setminus I$,
    \begin{equation}\label{eq:Feq0}
    \sum_{Y\in \cA_u} (-1)^{|Y|}q^{|Y|}\tR_{\theta_u(Y),y} = 0. \end{equation}
    If, furthermore, the hypercube cluster satisfies the numerical criterion, then for any $y\in [u,v]\setminus I$,
    \begin{equation}\label{eq:Funtildeeq0} 
    \sum_{Y\in \cA_u} (-1)^{|Y|}q^{\ell(u,\theta_u(Y))-|\Lambda(Y)|}(q-1)^{|Y|} R_{\theta_u(Y),y} = 0.\end{equation}
\end{lemma}
\begin{proof}
    \Cref{lem:Feq0relR} shows that
    \[ \sum_{y\in [u,v]} \tR_{u,y,I}(-q)\tR_{y,v}(q)=0. \]
    Hence combining this with \Cref{thm:EintervalrelR} proves (\ref{eq:Feq0}).
    By substituting $q^{1/2}-q^{-1/2}$ in the left hand side of (\ref{eq:Feq0}) and multiplying by $q^{\ell(u,y)/2}$, we find that
    \[ \sum_{Y\in \cA_u} (-1)^{|Y|}q^{(\ell(u,\theta_u(Y))-|Y|)/2}(q-1)^{|Y|} R_{\theta_u(Y),y} = 0. \]
    The numerical criterion implies that $q^{(\ell(u,\theta_u(Y))-|Y|)/2} = q^{\ell(u,\theta_u(Y))-|\Lambda(Y)|}$, which proves (\ref{eq:Funtildeeq0}).
\end{proof}

\begin{corollary}
\label{cor:strong-nc-E-implies-bbdvw}
    If $I$ is a diamond-closed order ideal of $[u,v]$ with a hypercube cluster at $u$ which is strong, satisfies the numerical criterion, and has property (E), then (\ref{eq:BBDVW}) holds for $I$.
\end{corollary}
\begin{proof}
Exchanging the order of summation in (\ref{eq:BBDVW}) the assertion becomes:
\[ \sum_{y\in [u,v]\setminus I} \left(\sum_{Y\in \cA_u} (-1)^{|Y|}q^{\ell(u,\theta_u(Y))-|\Lambda(Y)|}(q-1)^{|Y|} R_{\theta_u(Y),y}\right)P_{y,v} = 0. \]
By \Cref{lem:bbdvw-inner-sum-vanishes}, the inner sum vanishes for all $y \in [u,v] \setminus I$.
\end{proof}

Hence to prove \Cref{thm:BBDVWorig}, it is enough to prove \Cref{thm:intro-strong-and-nc}, so that \Cref{cor:strong-nc-E-implies-bbdvw} can be applied to any hypercube decomposition of a lower interval.

\subsection{Property (E)}
\label{sec:property-E}
The remainder of the paper is dedicated to establishing \Cref{def:strong-nc-join-E}(i-iii) for hypercube decompositions of lower intervals. There is in fact a larger class of intervals already known to satisfy property (E), introduced in \cite{elementary-paper}.

\begin{definition}
\label{def:simple}
    The interval $[u,v]$ is \emph{simple} if the roots $e_i-e_j$ corresponding to the cover relations $u\lessdot (i\,j)u \leq v$ are linearly independent. 
\end{definition}

The following was shown in the proof of \cite[Corollary 4.6]{elementary-paper}.
\begin{proposition} 
\label{prop:simpleisE}
    If $[u,v]$ is simple and $I$ is a principal diamond-closed order ideal, then $I$ satisfies property (E). 
\end{proposition}

Any lower interval is simple. Hence we deduce from \Cref{prop:simpleisE}:

\begin{corollary}\label{cor:lowerisE}
    If $[e,v]$ is a lower interval and $I$ is a hypercube decomposition of $[e,v]$, then $I$ satisfies property (E).
\end{corollary}

\Cref{thm:EintervalrelR} and \Cref{prop:simpleisE} imply another corollary of independent interest, which shows that relative $\tR$-polynomials of simple intervals can be computed combinatorially. (Note, though, that simple intervals cannot be recognized combinatorially.) 

\begin{corollary}
    If $[u,v]$ is a simple interval and $I$ is a principal diamond-closed order ideal with a strong hypercube cluster at $x\in I$, then 
    \[ \tR_{x,v,I} = \sum_{\substack{Y\in \cA_x \\ \theta_x(Y)=v}} q^{|Y|}. \]
\end{corollary}

\Cref{prop:simpleisE} and \Cref{lem:bbdvw-inner-sum-vanishes} together imply a strengthening of \Cref{thm:introsimple}.

\begin{corollary}
    If $[u,v]$ is a simple interval, and $I$ is a principal diamond-closed order ideal of $[u,v]$ with a strong hypercube cluster at $u$, then
    \[ \sum_{Y\in \cA_u}(-q)^{|Y|}\tR_{\theta_u(Y),v} = 0. \]
\end{corollary}

\section{Hypercube decompositions of lower intervals}
\label{sec:lower}

We now prove \Cref{thm:intro-strong-and-nc}, which we restate here.

\begin{theorem}
\label{thm:hcds-of-lower-intervals-are-nice}
    Let $v \in S_n$ and let $I=[e,z]$ be any hypercube decomposition of $[e,v]$. Then $I$ has property (E) and the hypercube cluster at $e$, relative to $I$:
    \begin{enumerate}[(i)]
        \item \label{item:strong} is strong, and
        \item \label{item:nc} satisfies the numerical criterion.
    \end{enumerate}
\end{theorem}

\subsection{Blocks of hypercube decompositions}
\label{sec:blocks}
In order to prove \Cref{thm:hcds-of-lower-intervals-are-nice}, we will study the structure of $\cA_e$, relative to a hypercube decomposition $I$ of $[e,v]$. We will see in this section that $I$ induces a partition of $[n]$ into \emph{blocks} and see in \Cref{sec:camels} that the structure of $\cA_e$ is controlled by certain combinatorial diagrams, decorating the blocks, that we call \emph{caravans}. Our starting point is the following fact:

\begin{proposition}
\label{prop:dc-order-ideal-of-lower-is-parabolic}
Let $v \in W=S_n$ and let $I$ be any diamond-closed order ideal in $[e,v]$. Then $I=[e,v] \cap W_J$ for some parabolic subgroup $W_J \subset W$. Moreover, $I$ is a principal order ideal.
\end{proposition}
\begin{proof}
Lower intervals are in particular \emph{simple} intervals (see \Cref{def:simple}). Thus, by \cite[Prop.~4.4]{elementary-paper}, $I=[e,v] \cap W'$ where $W'$ is the reflection subgroup of $W$ generated by the set $J$ of atoms of $I$. These atoms are all simple reflections, and so $W'=W_J$ is in fact a standard parabolic subgroup. Then \cite[Thm.~2.2]{billey-fan-losonczy} implies that $[e,v] \cap W_J$ has a unique maximal element under Bruhat order, so $I$ is principal.
\end{proof}

In particular, we have:

\begin{corollary}
\label{cor:hcd-of-lower-is-parabolic}
Let $v \in W=S_n$. Then any hypercube decomposition of $[e,v]$ is of the form $I=[e,v] \cap W_J$, where $J$ is the set of atoms of $I$.
\end{corollary}

\begin{remark}
Although $[e,v] \cap W_J$ is a principal diamond-closed order ideal for any choice of $v$ and $J$, it is not always a hypercube decomposition of $[e,v]$.
\end{remark}

The parabolic subgroups of $S_n$ are of the form $S_{B_1} \times \cdots \times S_{B_b}$, where $B_1 \sqcup \cdots \sqcup B_b = [n]$, each \emph{block} $B_a$ is a nonempty set of consecutive integers, and where $S_B$ denotes the group of permutations of the block $B$. We assume that we have numbered the blocks so that the elements of $B_a$ are less than those of $B_{a'}$ for $a<a'$. Note that some blocks may be singletons.

Now fix any hypercube decomposition $I$ of $[e,v]$ and let $B_1,\ldots,B_b$ be the blocks associated to the parabolic subgroup $W_J$, with $J$ the atoms of $I$. Since $I=[e,v] \cap W_J$ by \Cref{cor:hcd-of-lower-is-parabolic}, we note:
\begin{obs}
\label{obs:Ye-from-blocks}
The elements of $\cY_e$ are precisely the reflections $(i \: j) \leq v$ such that $i$ and $j$ belong to different blocks.
\end{obs}
We make the convention of writing all reflections as $(i \: j)$ with $i<j$. We now further analyze the structure of $\cY_e$.

\begin{lemma}
\label{lem:bruhat-on-reflections}
    We have $(i \: j) \leq (k \: l)$ in Bruhat order if and only if $k\leq i<j \leq l$.
\end{lemma}
\begin{proof} 
    The reverse implication follows immediately from the Subword Property for Bruhat order (see, e.g. \cite[Thm.~2.2.2]{Bjorner-Brenti}) after noting that the reduced word $s_i s_{i+1} \cdots s_{j-1} s_j s_{j-1} \cdots s_{i+1} s_i$ for $(i \: j)$ is a subword of the analogously constructed reduced word for $(k \: \ell)$. The forwards implication follows from the Tableau Criterion \cite[Theorem 2.6.3]{Bjorner-Brenti}.
\end{proof}

\begin{lemma}
\label{lem:no-S3}
    If $(i \: j) \in \cY_e$, then for all $k$ with $i<k<j$ exactly one of $(i \: k)$ and $(k \: j)$ lies in $\cY_e$.
\end{lemma}
\begin{proof}
    By \Cref{lem:bruhat-on-reflections}, $(i \: k)$ and $(k \: j)$ are incomparable in Bruhat order, and both are less than $(i \: j)$. It cannot be that both lie in $\cY_e$, since this would violate uniqueness in the hypercube cluster at $e$: we would have diamonds $e \rightrightarrows (i \: k), (k \: j) \rightrightarrows (i \: j \: k)$ and $e \rightrightarrows (i \: k), (k \: j) \rightrightarrows (i \: k \: j)$ since $(i \: k \: j), (i \: j \: k) \leq (i \: j)$. On the other hand, it cannot be that both lie in $I$, since this would imply that $(i \: j) \in I$ by diamond-closure.
\end{proof}

\begin{corollary}
\label{cor:consecutive-blocks}
    If $(i \: j) \in \cY_e$, then $i$ and $j$ lie in consecutive blocks.
\end{corollary}
\begin{proof}
    If not, then $\exists k$ with $i<k<j$ such that $k$ lies in a different block than both $i$ and $j$. By \Cref{obs:Ye-from-blocks} and \Cref{lem:bruhat-on-reflections}, we see that both $(i \: k)$ and $(k \: j)$ lie in $\cY_e$, but this contradicts \Cref{lem:no-S3}.
\end{proof}

\begin{lemma}
\label{lem:no-nested-edges}
    If $(i \: j) \in \cY_e$, then $i+1,\ldots,j-1$ lie in the same block.
\end{lemma}
\begin{proof}
    Suppose to the contrary that some $i',j'$ with $i<i'<j'<j$ lie in different blocks. Since $(i \: j) \in \cY_e$, we have in particular that $(i \: j) \leq v$, and so $(i' \: j') \leq v$ by \Cref{lem:bruhat-on-reflections}. Moreover, $(i' \: j') \in \cY_e$ by \Cref{obs:Ye-from-blocks}.

    Now, since $I$ is an order ideal and since $(i' \: j')$ is less than $(i \: j')$ and $(i' \: j)$ and both of these are less than $(i \: j)$ by \Cref{lem:bruhat-on-reflections}, it must be that $(i \: j'), (i' \: j) \in \cY_e$. By \Cref{lem:no-S3}, this implies that $(i \: i'), (j' \: j) \in I$, since $(i' \: j') \in \cY_e$. By diamond-closure, we see that the product $(i \: i')(j' \: j)$ also lies in $I$. There are four elements $y \leq (i\:j)$ with $(i \: i')(j' \: j) \to y$ in $\Gamma(e,v)$, namely $(i \: i' \: j \: j'), (i \: j' \: j \: i'), (i \: j \: j' \: i'),$ and $(i \: i' \: j' \: j)$, which are incomparable. All four of these elements are greater than $(i'\: j')$ and therefore do not lie in $I$ and so form an antichain in $\cA_{(i \: i')(j' \: j)}$. Thus these four elements span a $4$-hypercube $\mathcal{H}$ in $\Gamma(e,v)$; furthermore, $\mathcal{H}$ is contained in the reflection subgroup $W'=S_{\{i,i',j',j\}}$ by \cite[Lemma 3.1]{Dyer-bruhat-graph}. But one can easily check that the Bruhat graph of $W'$ (which by \cite{Dyer-bruhat-graph} can be identified with $\Gamma(e,(i\:j)(i'\:j'))$) contains only one $4$-hypercube $\mathcal{H}'$ and that $\mathcal{H}'$ has bottom vertex $(i' \: j')$ rather than $(i \: i')(j' \: j)$, a contradiction.
\end{proof}

\begin{corollary}
\label{cor:arcs-use-endpoint}
    Suppose $(i \: j) \in \cY_e$ and let $B_{a}$ and $B_{a+1}$ be the blocks containing $i$ and $j$, respectively. Then $i=\max(B_a)$ or $j=\min(B_{a+1})$.
\end{corollary}
\begin{proof}
    If not, then $i+1 \in B_a$ while $j-1 \in B_{a+1}$, contradicting \Cref{lem:no-nested-edges}.
\end{proof}

\subsection{Antichains, camels, and caravans}
\label{sec:camels}

We now extract the implications of the results of \Cref{sec:blocks} for the set of antichains $\cA_e$ in $(\cY_e, \leq)$. As before, we let $I$ be a fixed hypercube decomposition of a lower interval $[e,v] \subset W=S_n$.

We can construct a directed graph called the \emph{caravan} $\car(Y)$ of an antichain $Y \in \cA_e$, having vertices $[n]$ drawn in increasing order from left to right, grouped according to the blocks induced by $I$, and with an arc from $i$ to $j>i$ whenever $(i \: j) \in \cY_e$. \Cref{lem:bruhat-on-reflections} and \Cref{lem:no-nested-edges} guarantees that the arcs do not nest. \Cref{cor:consecutive-blocks,cor:arcs-use-endpoint} imply that each arc connects vertices in consecutive blocks and begins at the maximal vertex of its left block or ends at the minimal vertex of its right block (see \Cref{ex:camel}). We call the non-singleton connected components of a caravan \emph{camels}. Let $\preceq$ denote the total order on the set of camels according to their minimal vertices.

\begin{example}
\label{ex:camel}
Consider the interval $[e,361245]$ with $z=214356$ defining a hypercube decomposition $I=[e,z]$. Then 
\[
Y = \{(1,3),(2,4),(4,5),(5,6)\}
\]
is an antichain in $\cA_e$. The caravan $\car(Y)$ is depicted in \Cref{fig:camel}.

\begin{figure}
    \centering
    \begin{tikzpicture}
        \fill[pink, opacity=.5] (.75,-.5) rectangle (2.25,.5);
        \fill[pink, opacity=.5] (2.75,-.5) rectangle (4.25,.5);
        \fill[pink, opacity=.5] (4.75,-.5) rectangle (5.25,.5);
        \fill[pink, opacity=.5] (5.75,-.5) rectangle (6.25,.5);
        \node (1) at (1,0) {$1$};
        \node (2) at (2,0) {$2$};
        \node (3) at (3,0) {$3$};
        \node (4) at (4,0) {$4$};
        \node (5) at (5,0) {$5$};
        \node (6) at (6,0) {$6$};
        \draw[thick] (1.north) edge[bend left] (3.north);
        \draw[thick] (2.north) edge[bend left] (4.north) (4.north) edge[bend left] (5.north) (5.north) edge[bend left] (6.north);
    \end{tikzpicture}
    \caption{An example of a caravan with four blocks (shaded in red) and two camels. All edges are implicitly directed towards the right.}
    \label{fig:camel}
\end{figure}
\end{example}

\begin{proposition}
\label{prop:antichains-are-camels}
Let $v \in W=S_n$ and let $I$ be any hypercube decomposition of $[e,v]$, and let $Y \in \cA_e$. Then the caravan $\car(Y)$ has the following properties:
\begin{enumerate}
    \item \label{item:in-degree} Every vertex of $\car(Y)$ has in-degree and out-degree each at most $1$.
    \item \label{item:crossing-arcs} Suppose $(i \: j), (i' \: j') \in Y$ have $i<i'<j<j'$, then $j=\min(B_{a+1})$ and $i'=\max(B_a)$ for some $a$.
    \item \label{item:crossing-camels} Let $C \prec C'$ be camels of $\car(Y)$. Then either $\max(C)<\min(C')$ or $\max(C)=\min(C')+1$, in which case $\max(C)$ and $\min(C')$ are the minimal and maximal vertices in their blocks, respectively.
\end{enumerate}
\end{proposition}
\begin{proof}
Since $Y$ is an antichain, no vertex can have in-degree greater than $1$, since the two arcs would necessarily nest, and therefore by \Cref{lem:bruhat-on-reflections} the two reflections would be comparable; the same is true for out-degree. This proves (\ref{item:in-degree}).

Now suppose $(i \: j), (i' \: j') \in Y$ with $i<i'<j<j'$. By \Cref{cor:arcs-use-endpoint}, we have $i=\max(B_a)$ or $j=\min(B_{a+1})$ for some $a$. Since $j-i \geq 2$, we cannot have both. So suppose first that $i=\max(B_a)$. In this case, since arcs connect consecutive blocks, we have $i',j \in B_{a+1}$ and $j' \in B_{a+2}$. Thus $i'$ is not the maximum element of its block, so we must have $j'=\min(B_{a+2})$ by \Cref{cor:arcs-use-endpoint}. Now, since $(i \: j), (i' \: j') \in Y$, we must have $(i \: j)(i' \: j') \leq v$ by the hypercube condition, since this is the top of the unique diamond in $\Gamma$ containing $e, (i \: j),$ and $(i' \: j')$. By \Cref{obs:Ye-from-blocks}, we have $(i' \: j) \in I$, since $(i' \: j) < (i \: j) < v$ but $i'$ and $j$ lie in the same block. By \Cref{lem:no-S3} this implies that $(i \: i'), (j \: j') \in \cY_e$. 

There are edges $(i' \: j) \to y$ for each $y \in Y'=\{(i' \: j \: i), (i' \: i \: j), (i' \: j' \: j), (i' \: j \: j')\}$. Furthermore, these four elements are incomparable, each element $y$ lies below $(i \: j)(i' \: j')$, and hence in $[e,v]$, and each lies outside of $I$, since $I$ is an order ideal and $y \geq (i \: i') \not \in I$ or $y \geq (j \: j') \not \in I$. Thus there must be a $4$-hypercube spanned by these elements. As argued in the proof of \Cref{lem:no-nested-edges}, this hypercube must be contained in the subgraph of $\Gamma(e,v)$ on $S_{\{i,i',j,j'\}}\cong S_4$. An easy check shows that there is a unique 4-hypercube in $S_4$, and that the top of that hypercube is $(i\: j')$. Hence we must have $\theta_{(i' \: j)}(Y')=(i \: j')$, so in particular, $(i \: j') \in \cY_e$. But this contradicts \Cref{lem:no-S3}, since $(i \: i'), (i' \: j') \in \cY_e$. Thus it must instead be the case that $j=\min(B_{a+1})$, with $i \in B_a$. This implies that $i' \in B_a$ and so $j' \in B_{a+1}$ and $j'$ is not minimal in its block. Another application of \Cref{cor:arcs-use-endpoint} implies that $i'=\max(B_a)$, so we have proven (\ref{item:crossing-arcs}).

Now suppose that $C \prec C'$ are camels of $\car(Y)$ with $\max(C)>\min(C')\eqqcolon i'$. Let $y'=(i' \: j')$ be the leftmost arc of $C'$. By hypothesis, we have $\min(C)<i'<\max(C)$, so $i<i'<j$ for some arc $(i \: j)$ of $C$. Since $Y$ is an antichain, we must have $j'>j$, otherwise by \Cref{lem:bruhat-on-reflections} we would have $y' < (i \: j)$. By part (\ref{item:crossing-arcs}) we must have $j=\min(B_{a+1})$ and $i'=\max(B_a)$ for some $a$. Suppose that $(i \: j)$ is not the rightmost arc of $C$, so there is some $(j \: k)$ in $C$. By \Cref{cor:consecutive-blocks}, we have $k \in B_{a+2}$ and hence $i'<j<j'<k$. But now part (\ref{item:crossing-arcs}) implies that $j'=\min(B_{a+1})$, so $j=j'$, a contradiction. Thus  $(i \: j)$ is the rightmost arc of $C$ and so $j=\max(C)=\min(B_{a+1})=\max(B_a)+1=i'+1=\min(C')+1$, proving (\ref{item:crossing-camels}).
\end{proof}

\Cref{prop:antichains-are-camels}(\ref{item:crossing-camels}) tells us that the supports $[\min(C),\max(C)]$ for different camels $C$ are almost disjoint: the only possible overlaps are at the ends of the camels, where they might overlap on two elements. This means that the subsets of $Y \in \cA_e$ corresponding to different camels can almost be treated independently, with a small adjustment for the possible overlap. We say that the pair $y$ and $y'$ of arcs of $C$ and $C'$ as in \Cref{prop:antichains-are-camels}(\ref{item:crossing-arcs}) are \emph{crossing}. All other pairs of arcs are \emph{noncrossing}. We similarly say that camels $C$ and $C'$ are crossing or noncrossing according to whether they contain such a pair of arcs. We let $\cross(Y)$ denote the number of pairs of crossing camels in $\car(Y)$; note that, by \Cref{prop:antichains-are-camels}, any pair of crossing camels must be consecutive in the total order $\prec$. For example, the caravan $\car(Y)$ depicted in \Cref{fig:camel} has $\cross(Y) = 1$.

\begin{proposition}
\label{prop:lambda-facts}
Let $v \in W=S_n$ and let $I$ be any hypercube decomposition of $[e,v]$. 
\begin{enumerate}
    \item \label{item:singleton-lambda} For $y=(i \: j) \in \cY_e$, we have $|\Lambda(\{y\})|=j-i$.
    \item \label{item:intersection} Let $y=(i\:j)$ and $y'=(i' \: j')$ be incomparable elements of $\cY_e$, with $i<i'$. Then $\Lambda(\{y\}) \cap \Lambda(\{y'\})= \emptyset$ if $y,y'$ are noncrossing, and $\Lambda(\{y\}) \cap \Lambda(\{y'\})=\{s_{i'}\}$ if $y,y'$ are crossing (in which case, by \Cref{prop:antichains-are-camels}, we have $j=i'+1$).
\end{enumerate}
\end{proposition}
\begin{proof}
Let $y=(i \: j) \in \cY_e$. If $j-i=1$, then (\ref{item:singleton-lambda}) is clear. So assume $j-i>1$; by \Cref{cor:arcs-use-endpoint} we have $i=\max(B_a)$ or $j=\min(B_{a+1})$ (and not both). Suppose without loss of generality that $i=\max(B_a)$, so $j \neq \min(B_{a+1})$. For each $k=i+1,\ldots,j$, we have $(i \: k) \leq (i \: j)$ by \Cref{lem:bruhat-on-reflections} and so $(i \: k) \in \Lambda(\{y\})$ by \Cref{obs:Ye-from-blocks}. Furthermore, these are all of the elements of $\Lambda(\{y\})$ since the other candidates $(i' \: j')$ with $i<i'< j'\leq j$ are not in $\cY_e$ by \Cref{obs:Ye-from-blocks} and the fact that such $i',j'$ lie in the same block $B_{a+1}$. This proves (\ref{item:singleton-lambda}).

Now let $y=(i\:j),y'=(i' \: j') \in \cY_e$ be incomparable, with $i<i'$. Then $j' > j$. If $y,y'$ are noncrossing, then $i'\geq j$, and it is clear by the previous paragraph that $\Lambda(\{y\})$ and $\Lambda(\{y'\})$ are disjoint. Otherwise $i<i'<j<j'$, so $y$ and $y'$ are crossing and, by \Cref{prop:antichains-are-camels}, $j=i'+1$. In this case $\Lambda(\{y\})=\{(k \: j) \mid i \leq k < j\}$ and $\Lambda(\{y'\})=\{(i' \: k') \mid i' < k' \leq j'\}$, whose intersection is $\{(i' \: j)\}=\{s_{i'}\}$, proving (\ref{item:intersection}).
\end{proof}

\begin{lemma}
\label{lem:almost-reduced-product-gives-hypercube}
    Let $v \in W=S_n$, let $I$ be any hypercube decomposition of $[e,v]$, and let $Y \in \cA_e$. Let $\eta = y_1 \cdots y_{|Y|}$ be the product of the elements of $Y$ in some order. Then there is a hypercube subgraph of $\Gamma$ spanned by $Y$ with top element $\eta$ (note that $\eta$ may not lie in $[e,v]$).
\end{lemma}
\begin{proof}
    A reduced word for a reflection $t=(i \: j)$ contains simple reflections $s_i,\ldots,s_{j-1}$. These sets of simple reflections for $t$ in the same camel $Y_a$ are disjoint, so the product of the elements of $Y_a$ in any order is length-additive. The reflections appearing in different camels commute, so we may reorder the product $\eta$ as a product from left to right of its induced subproducts on each camel. The sets of simple reflections appearing in each of these factors are now disjoint except that there is a single simple reflection shared by any camels which cross, so the product is still increasing in length. Therefore deleting any term in the product yields an $\eta'$ with $\ell(\eta')<\ell(\eta)$ so $\eta' \to \eta$. Finally, note that each $y_a$ contains some simple reflection in its support not contained in the support of any other $y_{a'}$. This ensures that deleting two different subsets of the terms of the product yields different elements. 
\end{proof}

\begin{proposition}
\label{prop:theta-facts}
Let $v \in W=S_n$ and let $I$ be any hypercube decomposition of $[e,v]$. Let $Y \in \cA_e$ and let $Y=Y_1 \sqcup \cdots \sqcup Y_c$ be its decomposition according to the camels of $\car(Y)$.
\begin{enumerate}
    \item \label{item:connected-component} For each $1 \leq a \leq c$, $\theta_e(Y_a)$ is given by the product of the elements of $Y_a$ in some order.
    \item \label{item:whole-antichain} The $\{\theta_e(Y_a)\}$ commute and $\theta_e(Y) = \prod_{a=1}^c \theta_e(Y_a)$. 
\end{enumerate} 
\end{proposition}
\begin{proof}
We prove (\ref{item:connected-component}) by induction on $|Y_a|$. 
 
If $|Y_a|=2$, then $Y_a=\{(i_1 \: i_2), (i_2 \: i_3)\}$ for some $i_1<i_2<i_3$. It is easy to see that there are exactly two diamonds in $\Gamma$ spanned by $Y_a$: one with top $(i_1 \: i_2)(i_2 \: i_3)$ and the other with top $(i_2 \: i_3)(i_1 \: i_2)$. Since $I$ is a hypercube decomposition, exactly one of these two elements lies in $[e,v]$ and this element is by definition $\theta_e(Y_a)$. Now suppose that $|Y_a| \geq 3$. For each $y \in Y_a$ we have by induction that $\theta_e(Y_a \setminus \{y\})$ is a product of $Y_a \setminus \{y\}$ in some order. Moreover, we claim that these products are compatible in the sense that there is a product $\eta$ of the elements of $Y$ with $\theta_e(Y_a \setminus \{y\}) \to \eta$ for all $y \in Y_a$ (\emph{a priori} $\eta$ need not be in $[e,v]$). To see this, define a relation $\mathscr{R}$ on $Y_a$ with $(\zeta,\zeta') \in \mathscr{R}$ if $\zeta=\zeta'$ or if $\zeta, \zeta'$ do not commute and $\zeta$ appears before $\zeta'$ in one of the products $\theta_e(Y_a \setminus \{y\})$. Then $\mathscr{R}$ is antisymmetric since if $\zeta,\zeta'$ appeared in different orders in two of these products then \Cref{lem:almost-reduced-product-gives-hypercube} would imply that $\zeta\zeta', \zeta'\zeta \leq v$, contradicting the hypothesis that $I$ is a hypercube decomposition. Furthermore, the transitive closure $\mathscr{P}$ of $\mathscr{R}$ is a poset since if we have $(\zeta_1,\zeta_2),(\zeta_2,\zeta_3), \ldots, (\zeta_{k-1},\zeta_k) \in \mathscr{R}$ for $k \geq3$ distinct elements $\zeta_1,\ldots,\zeta_k$, then it follows by \Cref{prop:antichains-are-camels} that $\zeta_1,\zeta_k$ commute, so $(\zeta_k,\zeta_1) \not \in \mathscr{R}$. Thus, taking any linear extension of $\mathscr{P}$ we find a compatible product $\eta$. On the other hand, $\theta_e(Y_a \setminus \{y\}) \to \theta_e(Y_a)$ for all $y \in Y_a$. If $\eta \neq \theta_e(Y_a)$, then we have a directed complete bipartite subgraph $K_{3,2}$ of $\Gamma$. But by \cite[Lemma 3.1]{Dyer-bruhat-graph}, no such subgraph of the Bruhat graph of $S_n$ exists: if there were one, then there would be one in a rank-two reflection subgroup $S_3$ or $S_1 \times S_1$, but an easy check shows that there is not. Thus $\eta=\theta_e(Y_a)$ and we have proven (\ref{item:connected-component}). 

Let $y \in Y_a$ and $y' \in Y_{a'}$, for $a \neq a'$. Then $y$ and $y'$ commute since no element of $[n]$ is moved by both permutations. Thus, by (\ref{item:connected-component}), $\theta_e(Y_a)$ and $\theta_e(Y_{a'})$ commute. The fact that $\theta_e(Y) = \prod_{a=1}^c \theta_e(Y_a)$ can be argued as in part (\ref{item:connected-component}).
\end{proof}

\begin{proposition}
\label{prop:theta-lengths}
Let $v \in W=S_n$ and let $I$ be any hypercube decomposition of $[e,v]$. Let $Y \in \cA_e$ and let $Y=Y_1 \sqcup \cdots \sqcup Y_c$ be its decomposition according to the camels $C_1 \prec \cdots \prec C_c$ of $\car(Y)$. 
\begin{enumerate}
    \item \label{item:connected-component-length} If $Y_a=\{(i_1 \: i_2), (i_2 \: i_3), \cdots, (i_{d} \: i_{d+1})\}$, then $\ell(\theta_e(Y_a))=2(i_{d+1}-i_1)-d$.
    \item \label{item:whole-antichain-length} $\ell(\theta_e(Y))=\sum_{a=1}^c \ell(\theta_e(Y_a))-2\cross(Y)$.
\end{enumerate} 
\end{proposition}
\begin{proof}
    Arguing as in \Cref{lem:almost-reduced-product-gives-hypercube} and using \Cref{prop:theta-facts}(\ref{item:connected-component}) we have 
    \[
    \ell(\theta_e(Y_a))=\sum_{d'=1}^d \ell((i_{d'} \: i_{d'+1})) = \sum_{d'=1}^d (2(i_{d'+1}-i_{d'})-1) = 2(i_{d+1}-i_1)-d,
    \]
    proving (\ref{item:connected-component-length}).

    Now, we have $\theta_e(Y) = \prod_{a=1}^c \theta_e(Y_a)$ by \Cref{prop:theta-facts}(\ref{item:whole-antichain}). Since this product commutes, we can take it in the order $\theta_e(Y_1) \cdots \theta_e(Y_c)$. The sets of simple reflections appearing in consecutive factors are now disjoint if the corresponding camels $C_i$ and $C_{i+1}$ are not crossing. If $C_i$ and $C_{i+1}$ are crossing, then by \Cref{prop:antichains-are-camels}(\ref{item:crossing-camels}) there is a unique simple reflection appearing in both factors, which is the unique element of $D_R(\theta_e(Y_i)) \cap D_L(\theta_e(Y_{i+1}))$ (where $D_R$ and $D_L$ denote right and left descent sets, respectively). Canceling this squared simple reflection in the product yields (\ref{item:whole-antichain-length}). 
\end{proof}

\subsection{Proof of \Cref{thm:hcds-of-lower-intervals-are-nice}}

We are now ready to prove \Cref{thm:hcds-of-lower-intervals-are-nice}, the main result of \Cref{sec:lower}.

\begin{proof}[Proof of \Cref{thm:hcds-of-lower-intervals-are-nice}]
Let $v \in W=S_n$ and fix a hypercube decomposition $I$ of $[e,v]$. Let $Y \in \cA_e$ be an antichain.

    We first prove (\ref{item:strong}). Suppose that $\Gamma(e,v)$ has a subgraph:
    \begin{equation}
    \label{eq:proof-strong-diamond}
        \begin{tikzcd}[cramped,sep=small] & w & \\ \theta_e(Y \cup \{y\})\ar[ur]& &\theta_e(Y \cup \{y'\})\ar[ul] \\ &\theta_e(Y)\ar[ul,"t"]\ar[ur,"t'"']& \end{tikzcd},
    \end{equation}
    where $Y \cup \{y\}, Y \cup \{y'\} \in \cA_e$. We need to show that $Y \cup \{y,y'\} \in \cA_e$ and that $\theta_e(Y \cup \{y,y'\})=w$.

    Suppose that $Y \cup \{y,y'\} \not \in \cA_e$, so, without loss of generality, we must have $y<y'$. By \Cref{lem:bruhat-on-reflections} and \Cref{cor:arcs-use-endpoint}, these elements are of the form $y=(i\:j)$ and $y'=(i \: j')$ for some $i<j<j'$ (or $y=(i\:j)$ and $y'=(i' \: j)$ with $i'<i<j$ which is analogous). By \Cref{prop:theta-facts} we know that $\theta_e(Y \cup \{y'\})=y_1 \cdots y_a y' y_{a+1} \cdots y_{|Y|}$ for some ordering $y_1,\ldots, y_{|Y|}$ of $Y$ and some $a$. This word is length-additive with the exception of the cancellation of some descents of consecutive reflections, as discussed in the proof of \Cref{prop:theta-lengths}. By the subword property (after accounting for this cancellation) we see that $\eta \coloneqq y_1 \cdots y_a y y_{a+1} \cdots y_{|Y|} < \theta_e(Y \cup \{y'\}) < v$. By \Cref{lem:almost-reduced-product-gives-hypercube}, we conclude that $\eta$ is the top of a hypercube subgraph of $\Gamma(e,v)$ spanned by $Y \cup \{y\}$. By the uniqueness of hypercubes, this implies that $\theta_e(Y \cup \{y\})=\eta < \theta_e(Y \cup \{y'\})$. A case check shows that any diamond subgraph of $\Gamma$ whose two middle vertices are comparable in Bruhat order must be spanned by reflections $(i_1 \: i_4)$ and $(i_2 \: i_3)$ with $i_1<i_2<i_3<i_4$ and $x(i_1)<x(i_2)<x(i_3)<x(i_4)$, where $x$ is the bottom vertex of the diamond. Applying this to the diamond (\ref{eq:proof-strong-diamond}), yields a contradiction, as the reflections $t,t'$ spanning this diamond are conjugates of $y$ and $y'$ by the same permutation (since $\theta_e(Y)=y_1\cdots y_{|Y|}$, again by uniqueness of hypercubes) and hence do not commute. Thus $Y \cup \{y,y'\} \in \cA_e$.

    Now, the only elements $\pi\in W$ with $\theta_e(Y \cup \{y\}),\theta_e(Y \cup \{y'\}) \to \pi$ are $tt'\theta_e(Y)$ and $t't\theta_e(Y)$; furthermore, one of these elements is $\theta_e(Y \cup \{y,y'\})$. Hence if $t,t'$ commute, we are done. Otherwise, by \Cref{prop:theta-facts} we have
    \begin{equation}
    \label{eq:theta-Y-cup-yyprime}
        \theta_e(Y \cup \{y,y'\})=y_1\cdots y_a y y_{a+1} \cdots y_b y' y_{b+1} \cdots y_{|Y|}
    \end{equation}
    where $y_1,\ldots,y_{|Y|}$ is some ordering of $Y$ and we have assumed without loss of generality that $y$ appears before $y'$. Suppose also that we have chosen the product (\ref{eq:theta-Y-cup-yyprime}) so as to minimize $b-a$. \Cref{lem:almost-reduced-product-gives-hypercube} implies that $\theta_e(Y \cup \{y'\}), \theta_e(Y \cup \{y\}),$ and $\theta_e(Y)$ are obtained by omitting $y, y',$ and both in the product, respectively. So $\theta_e(Y \cup \{y,y'\})=tt'\theta_e(Y)$ and $t=y_1\cdots y_ayy_a \cdots y_1$ and $t'=y_1\cdots y_by'y_b\cdots y_1$.

    Suppose that $w\neq \theta_e(Y \cup \{y,y'\})$, so $w=t't\theta_e(Y) \leq v$. It then follows from the Tableau Criterion \cite[Theorem 2.6.3]{Bjorner-Brenti} (arguing e.g. as in \cite[Lem.~5.3]{gaetz-gao-automorphism}) that the top element $tt't\theta_e(Y)=t'tt'\theta_e(Y)$ also lies in $[e,v]$. Since $b-a$ is minimized, $y,y_{a+1},\ldots,y_b,y'$ must correspond to a consecutive (either left-to-right or right-to-left) sequence of arcs in the same camel of $\car(Y\cup\{y,y'\})$. We compute
    \begin{align*}
        t'tt'\theta_e(Y)&=y_1 \cdots y_a (y_{a+1}\cdots y_by'y_b\cdots y_{a+1}yy_{a+1}\cdots y_by')y_{b+1}\cdots y_{|Y|} \\
        &\geq y_{a+1}\cdots y_by'y_b\cdots y_{a+1}yy_{a+1}\cdots y_by' \eqqcolon v'.
    \end{align*}
    where the inequality follows since each element of $Y \cup \{y,y'\}$ has in its support some simple reflection not in the support of any of the others (as in the proof of \Cref{lem:almost-reduced-product-gives-hypercube}). Using that $y,\ldots,y'$ are consecutive in the same camel and comparing one-line notations, we see $v \geq v' \geq yy_{a+1}, y_{a+1}y$ (or $v \geq v' \geq yy', y'y$ when $b-a=0$). This gives two hypercubes spanned by $\{y,y_{a+1}\}$ in $[e,v]$, violating uniqueness of hypercubes, a contradiction. This completes the proof of (\ref{item:strong}). 
    
    We now prove (\ref{item:nc}).  Let $Y_1 \sqcup \cdots \sqcup Y_c$ be the decomposition of $Y$ according to the camels $C_1 \prec \cdots \prec C_c$ of $\car(Y)$. Since the arcs within each camel $Y_a$ are noncrossing, \Cref{prop:lambda-facts} implies that $|\Lambda(Y_a)|=\max(C_a)-\min(C_a)$ (the result of a telescoping sum). On the other hand, by \Cref{prop:theta-lengths}(\ref{item:connected-component-length}) we have
    \[
    L_a \coloneqq \frac{\ell(\theta_e(Y_a))+|Y_a|}{2}=\max(C_a)-\min(C_a)=|\Lambda(Y_a)|.
    \]
    Then by \Cref{prop:lambda-facts}(\ref{item:intersection}) and \Cref{prop:antichains-are-camels}(\ref{item:crossing-camels}) we see that
    \[
    |\Lambda(Y)|=\sum_a |\Lambda(Y_a)| - \cross(Y)
    \]
    and by \Cref{prop:theta-lengths}(\ref{item:whole-antichain-length}) we have
    \[
    \frac{\ell(\theta_e(Y))+|Y|}{2}=\frac{\left(\sum_{a=1}^c \ell(\theta_e(Y_a))-2\cross(Y)\right)+\sum_{a=1}^c |Y_a|}{2}=\sum_{a=1}^c L_a - \cross(Y)=|\Lambda(Y)|.
    \]
    This proves (\ref{item:nc}).

    Finally, $I$ has property (E) by \Cref{cor:lowerisE}.
\end{proof}

\begin{proof}[Proof of \Cref{thm:BBDVWorig}]
The theorem follows immediately from \Cref{thm:hcds-of-lower-intervals-are-nice} and \Cref{cor:strong-nc-E-implies-bbdvw}.
\end{proof}

\section*{Acknowledgments}
GTB was supported by National Science Foundation grants DMS-2152991 and DMS-2503536. CG was partially supported by NSF grant DMS-2452032 and by a travel grant from the Simons Foundation. We are grateful to Francesco Brenti, Francesco Esposito, and Mario Marietti for interesting discussions during the workshop ``Bruhat order: recent developments and open problems" held at the University of Bologna, and to the anonymous referees for their helpful comments.

\bibliographystyle{halpha-abbrv}
\bibliography{arxiv-v3}
\end{document}